\documentclass{amsart}[11pt]
\usepackage{amsmath}
\usepackage{mathtools}
\usepackage{amsfonts}
\usepackage{mathrsfs}
\usepackage{amsthm}
\usepackage[a4paper]{geometry}

\usepackage[nosort]{cite}
\usepackage{enumitem}

\usepackage{cite}
\usepackage{cleveref}
\title{Hypocoercivity in Phi-entropy for the Linear Relaxation Boltzmann Equation on the Torus.}
\author{Josephine Evans}
\address{Department of Pure Mathematics and Mathematical Statistics\\
University of Cambridge\\
Wilberforce Road\\
Cambridge CB3 0WA, UK}
\email{jahe2@cam.ac.uk}
\thanks{The author was supported by the UK Engineering and Physical
Sciences Research Council (EPSRC) grant EP/H023348/1 for the
University of Cambridge Centre for Doctoral Training, the Cambridge
Centre for Analysis.}

\newtheorem{thm}{Theorem}
\newtheorem{defn}{Definition}
\newtheorem{lemma}{Lemma}

\newtheorem*{remark}{Remark}
\begin{document}

\maketitle
\begin{abstract}
This paper studies convergence to equilibrium for the spatially inhomogeneous linear relaxation Boltzmann equation in Boltzmann entropy and related entropy functionals, the $p$-entropies. Villani proved in \cite{V09} entropic hypocoercivity for a class of PDEs in a H\"{o}rmander sum of squares form. It was an open question to prove such a result for an operator which does not share this form. We prove a closed entropy-entropy production inequality \`a la Villani which implies exponentially fast convergence to equilibrium for the linear Boltzmann equation with a quantitative rate. The key new idea appearing in our proof is the use of a total derivative of the entropy of a projection of our solution to compensate for an error term which appears when using non-linear entropies. We also extend the proofs for hypocoercivity for the linear relaxation Boltzmann to the case of $\Phi$-entropy functionals.
\end{abstract}
\tableofcontents
\section{Introduction} In this paper we constructively prove convergence to equilibrium for the linear relaxation Boltzmann equation on the torus in relative entropy. We also look at other entropy functionals, the $\Phi$-entropies specifically $p$-entropies. The equation is
\begin{align}  \label{eq:bigeqf} 
\partial_t f +v \cdot \nabla_x f = \lambda \tilde{\Pi}(f) - \lambda f,
\end{align} where $f=f(t,x,v) : \mathbb{R}_{+} \times \mathbb{T}^d \times \mathbb{R}^d \rightarrow \mathbb{R}$ and $\lambda$ is a positive constant. We always consider $f$ to be a probability density so it is positive and of mass one, this is well known to be preserved by the equation. It is straightforward to show that this equation is well posed in $L^1$. The operator $\tilde{\Pi}$ is defined by
\[ \tilde{\Pi}(f)  =: \left(\int_{\mathbb{R}^d} f(t,x,u) \mathrm{d}u\right) \mathcal{M}(v), \] \[ \mathcal{M}(v) := (2\pi)^{-d/2} \exp \left( -\frac{|v|^2}{2}\right). \]The equilibrium state of this equation is $\mu(x,v) = \mathcal{M}(v)$. We give two separate notations here to emphasize when we consider it as a function of $v$ alone or a function of $x$ and $v$. We will always work in terms of $h=f/\mu$ which satisfies,
\begin{align}\label{eq:bigeqh}
\partial_t h + v\cdot \nabla_x h = \lambda\Pi h - \lambda h, 
\end{align} here we define $\Pi$ by
\[ \Pi h = \int_{\mathbb{R}^d} h(t,x,u) \mathcal{M}(u) \mathrm{d}u. \] So the function $\Pi h$ does not depend on $v$.

We want to study the convergence to equilibrium for solutions to equation \eqref{eq:bigeqf} in relative entropy, $H$, and Fisher information, $I$, of $f$ to $\mu$. Studying the relative entropy has been an important way of showing convergence to equilibrium for kinetic equations since Boltzmann's $H$-theorem \cite{B64}. Fisher information was introduced into kinetic theory by McKean to study convergence to equilibrium for a caricature of the Boltzmann equation \cite{M66}. These quantities are defined in terms of $h = f/\mu$, and are
\begin{align*}
H(h) =& \int_{\mathbb{T}^d \times \mathbb{R}^d} h \log(h) \mathrm{d}\mu,\\
I(h) =& \int_{\mathbb{T}^d \times \mathbb{R}^d} \frac{|\nabla h|^2}{h} \mathrm{d}\mu.
\end{align*}

\subsection{Previous work}
Villani and Desvillettes demonstrated convergence to equilibrium in weighted $H^1$ for spatially inhomogeneous kinetic equations including the Boltzmann equation in \cite{DV01, DV05}, their techniques were also applied to the linear Boltzmann equation in \cite{CCG03} where they show convergence faster than any power of $t$. After this the theory of hypocoercivity was developed and the equation is shown to converge to equilibrium in weighted $L^2$, \cite{H07}, by H\'{e}rau in order to demonstrate the applicability of the tools used in \cite{HN04}. Convergence in weighted $H^1$ is also demonstrated in section 5.1 of \cite{NM06}, by Mouhot and Neumann as a consequence of a more general theorem. The techniques used in both these papers exploit commutator relations between the transport and collision part of the equation using the tools of hypocoercivity also see \cite{EH03, HN04, HN05} for hypoellipticity based approaches; \cite{ V09} for Villani's method based on these earlier works; \cite{DMS15} for work directly in weighted $L^2$ spaces developing methods similar to \cite{H07} to extend the work to a wider classe of operators and \cite{MM14} which extends these results to a wider class of function spaces. The paper \cite{AAC15}, shows convergence in weighted $L^2$ spaces with improved rates, and studies the convergence in relative entropy for models with discrete velocities. A linearized version of the non-linear equation in the multi-species case is studied in \cite{AAC18} and a similar problem for the Elipsoidal BGK model is considered in \cite{Y15}. 

In these references, convergence is shown for $h$ in either $L^2(\mu)$ or $H^1(\mu)$. These norms control relative entropy, $H(\mu)$. Therefore these results do imply exponential convergence of relative entropy. This fact is written explicitly in \cite{HN04}, Corollary 1.2, in this paper we give a different result with a close form estimate. This means we do not require the initial data, $h_0$, to be in $L^2(\mu)$, which would exclude $f_0$ having heavy tails.

 The convergence demonstrated in all these papers is of the form
\begin{equation} \label{eq:hypo} \mathcal{E}(f(t)| \mu ) \leq C e^{- \gamma t} \mathcal{E}( f(0) | \mu ), \end{equation} where $\mathcal{E}$ is a functional or norm and $C$ and $\gamma$ are explicit constants. If $C=1$ the equation would be coercive in this norm. When $C>1$, we use the terminology introduced in \cite{V09} and say that it is hypocoercive. 

We now show briefly why our equation is not coercive. Let $\mathcal{A}$ be some set of functions. First we not that if \eqref{eq:hypo} holds with $C=1$ for every initial data $f(0) \in \mathcal{A}$ then this is equivalent to a functional inequality. Lets define another functional by
\[ D(f_0| \mu) = - \frac{\mathrm{d}}{\mathrm{d}t}_{t=0} \mathcal{E}(f(t)|\mu). \] Then if \eqref{eq:hypo} holds with $C=1$ for every $f(0) \in \mathcal{A}$ this is equivalent to 
\begin{equation} \label{eq:diffiq} D(f)  \geq \gamma \mathcal{E}(f|\mu),\end{equation} for every $f \in \mathcal{A}$.
We can see that \eqref{eq:diffiq} implies \eqref{eq:hypo} with $C=1$ by using Gr\"{o}nwall's inequality. Conversely if \eqref{eq:hypo} holds with $C=1$ and $f(0)$ then differentiating this inequality at $t=0$ will allow us to recover \eqref{eq:diffiq}.
We can check that this last inequality does not hold for the functionals we consider when $f$ is in local equilibrium (i.e. of the form $\rho (x) \mathcal{M}(v)$). More precisely we can check that $D(\rho \mathcal{M}|\mu) =0$. 

\subsection{Entropic hypocoercivity}
Studying such equations in relative entropy was introduced by Villani in \cite{V09}. More recently entropic hypocoercivity and hypocoercivity in different $\Phi$ entropies have been studied for diffusion operators \cite{B13, BLMV14, M15, B16}. Whilst most hypocoercivity theory has been done in $L^2(\mu^{-1}), H^1(\mu^{-1})$ there are several motivations to try and push the theory in the context of relative entropy. 
\begin{itemize}
\item We can enlarge the space of initial data for which we can show exponentially fast convergence to equilibrium. If we show a result for $f \in L^2(\mu^{-1})$ then we are constrained to work with initial data in $L^2(\mu^{-1})$. This means that $f_0$ must decay very fast at infinity. However, if $\mu = \exp (-|v|^2/2 + U(x))$ then we have
\[ H_\mu(f) = \int f\log(f/\mu) \mathrm{d}x \mathrm{d}v = \int f \log(f) \mathrm{d}x \mathrm{d}v + \int f ( |v|^2/2 + U(x)) \mathrm{d}x \mathrm{d}v.  \] Similarly, for Fisher information we have
\[ I_\mu(f) \leq I(f) + \int f |\nabla( |v|^2/2 + U(x))|^2\mathrm{d}v \mathrm{d}x.  \] So these quantities will be finite provided we have some moment bounds (depending on $U(x)$) and finite entropy and Fisher information. This is true for many distributions which decay only polynomially at infinity.
\item If we want to eventually study non-linear equations then it is often the case that strong spaces like $L^2(\mu^{-1})$ will not be a natural space for the equation. For initial data which is neither small nor close to the Maxwellian there is no well posedness theory for the Boltzmann equation in Hilbert spaces weighted against the equilibrium This problem is solved in the context of the Boltzmann equation by combining linearised theory with enlarging the space of solutions \cite{GMM17} and Desvillettes-Villani results to show when the solution will enter the linearised regime.
\item The relative entropy and relative Fisher information functionals behave well with respect to the dimension of the phase space that the equation is set in.  More specifically, suppose that $F_N = f^{\otimes N}$ then we have
\[ H(F_N) = \int f^{\otimes N}(z) \sum_i \log (f(z_i)) \mathrm{d}z = \sum_i \int f(z_i) \log(f(z_i)) \mathrm{d}z_i = N H(f). \]
We can also show that if $\Pi_1(F_N)$ is its first marginal,  and the particles are indistinguishable then
\[ H(\Pi_1(F_N)) \leq \frac{1}{N} H(F_N). \] Therefore, if we know that for all $N$ that
\[ H(F_N(t)) \leq Ce^{-\lambda t} H(F_N(0)), \] then we have that
\[ H(\Pi_1(F_N(t))) \leq \frac{C}{N} e^{-\lambda t} H(F_N(0)). \] Furthermore if $F_N(0)$ is a tensor product or similar we will have
\[ H(\Pi_1 (F_N(t))) \leq C e^{-\lambda t} \] where $C$ does not depend on $N$. Therefore the rates of convergence to equilibrium are uniform in $N$. On the other hand for $L^2$ the distance $\|F_N\|_2$ behaves like $\|\Pi_1 F_N\|_2^N$. So if we try the same computation we get that 
\[ \| \Pi_1 F_N(t) \|_2 \leq C e^{-\lambda t /N}. \] This effect becomes particularly important if one wishes to study particle systems and derive convergence results which are uniform in the number of particles. Entropic hypocoercivity has been used in \cite{OL17} to show convergence to the limit equation for oscillator chains.
\end{itemize}


\subsection{Villani's method for operators in H\"{o}rmander form and the problem for the linear relaxation Boltzmann equation}

The main purpose of this work is to demonstrate that entropic hypocoercivity can be proved for an equation which is not in `$A^*A +B$' form where $A, B$ are first order differential operators. The key difference between the proofs given here and those of previous hypocoercivity results arises because we do not have a diffusion operator. In order to understand this it is useful to compare the linear relaxation Boltzmann equation with the kinetic Fokker-Planck equation on the torus.
\begin{equation}\label{eq:kfp} \partial_t f + v\cdot \nabla_x f = \nabla_v \cdot \left( \nabla_v f + v f \right). \end{equation} Here we put $x \in \mathbb{T}^d, v \in \mathbb{R}^d$ as with the linear relaxation Boltzmann equation. This equation also has the same equilibrium $\mu$. Therefore we can write an equation on $h=f/\mu$ in the same way
\begin{equation} \label{eq:kfph} \partial_t h + v \cdot \nabla_x h = \left( \nabla_v - v \right) \cdot \nabla_v h. \end{equation} We can look at the dissipation of $H(h)$ for both these equation. We have, 
\begin{align}
D_{kFP}(h) =& \int_{\mathbb{T}^d \times \mathbb{R}^d} \frac{|\nabla_v h|^2}{h} \mathrm{d}\mu,\\
D_{LRB}(h) =& \int (h-\Pi h) \log(h) \mathrm{d}\mu.
\end{align} Here $D_{kFP}$ is the dissipation of relative entropy for the kinetic Fokker-Planck equation, \eqref{eq:kfph}, and $D_{LRB}$ is the dissipation of relative entropy for the linear relaxation Boltzmann equation, \eqref{eq:bigeqh}. We can see that both these quantities will vanish when $h$ is a function only of $x$. This is the local equilibria mentioned above. Here, we also see a crucial difference. The regularizing effect of the Fokker-Planck operator means that $D_{kFP}$ is a Fisher Information type term in the sense that it is of order one in terms of derivatives, where as $D_{LRB}$ is an entropy type term in the sense that it is order zero in terms of derivatives.

The proofs in \cite{HN04, V09, NM06} use crucially the way that the free transport operator interacts with mixed $x$ and $v$ derivatives of the solution. In the context of relative entropy and Fisher information we can state this precisely. If $h(t,x,v)$ is a solution to the free transport equation
\[ \partial_t h + v \cdot \nabla_x h =0,\] then we have that
\begin{align*}
\frac{\mathrm{d}}{\mathrm{d}t} \int_{\mathbb{T}^d \times \mathbb{R}^d} \frac{\nabla_x h \cdot \nabla_v h}{h} \mathrm{d} \mu = - \int_{\mathbb{T}^d \times \mathbb{R}^d} \frac{|\nabla_x h|^2}{h} \mathrm{d}\mu.
\end{align*}

In the case of the kinetic Fokker Planck equation Villani uses a functional which involves both the terms
\[ \int_{\mathbb{T}^d \times \mathbb{R}^d} h \log(h) \mathrm{d}\mu, \quad \mbox{and} \quad \int_{\mathbb{T}^d \times \mathbb{R}^d} \frac{\nabla_x h \cdot \nabla_v h}{h} \mathrm{d}\mu. \] These will then produce the terms
\[ \int_{\mathbb{T}^d \times \mathbb{R}^d} \frac{|\nabla_x h|^2}{h} \mathrm{d}\mu, \quad \mbox{and} \quad \int_{\mathbb{T}^d \times \mathbb{R}^d} \frac{|\nabla_v h|^2}{h} \mathrm{d}\mu,\] in their dissipation. In fact, Villani uses a functional of the form
\[ \mathscr{F}(h) = H(h) + \int_{\mathbb{T}^d \times \mathbb{R}^d} \frac{ a |\nabla_x h|^2 + 2b \nabla_x \cdot \nabla_v h + c |\nabla_v h|^2}{h} \mathrm{d}\mu. \] The other Fisher information type terms are required to make $\mathscr{F}$ a positive functional. There will be a lot of error terms in the dissipation as well as the useful terms appearing above. We can differentiate $\mathscr{F}(h)$ along the flow of the kinetic Fokker-Planck equation \eqref{eq:kfph}. Here we give the calculations briefly and refer to \cite{V09} for more detail. If we choose $b^2 <ac$ one can verify that
\begin{align*}
\frac{\mathrm{d}}{\mathrm{d}t} \mathscr{F}(h) \leq & - \left(1 +2c\right)\int_{\mathbb{T}^d \times \mathbb{R}^d}\frac{|\nabla_v h|^2}{h} \mathrm{d}\mu \\
& - 2b \int_{\mathbb{T}^d \times \mathbb{R}^d} \frac{|\nabla_x h|^2}{h} \mathrm{d}\mu \\
& - \left(2c +2b \right) \int_{\mathbb{T}^d \times \mathbb{R}^d} \frac{\nabla_x h \cdot \nabla_v h}{h} \mathrm{d}\mu.
\end{align*} Now if $a,b,c$ are sufficiently small we can split up the last term by the Cauchy-Schwartz inequality. We want to control this term by a large amount of the Fisher information terms with gradients in $v$ and a small amount of Fisher information terms with gradients in $x$. For appropriately chosen constants $a,b,c,$ this will give us that
\begin{align} \label{eq:dissiq}
\frac{\mathrm{d}}{\mathrm{d}t} \mathscr{F}(h) \leq - b  \int_{\mathbb{T}^d \times \mathbb{R}^d} \frac{|\nabla_x h|^2 + |\nabla_v h|^2}{h} \mathrm{d}\mu.
\end{align} This strategy for choosing $a,b$ and $c$ relies on the fact that the dissipation of entropy for the kinetic Fokker-Planck term is a Fisher information type term and can be used to control other Fisher information type terms appearing as errors. Now the goal is to compare the dissipation of $\mathscr{F}$ to $\mathscr{F}$ via a functional inequality. For this we need another tool, the logarithmic-Sobolev inequality. 
\begin{defn}\label{def:LS}
A measure $\mu$ in the space of positive measures on a state space $\Omega$ satisfies a logarithmic Sobolev inequality if for all $h$ we have that
\[ \int_\Omega \Big( h(z) \log(h(z)) -h(z) +1 \Big) \mathrm{d}\mu \leq C_{LS} \int_\Omega \frac{|\nabla_z h(z)|^2}{h(z)} \mathrm{d}\mu. \]
\end{defn}

We know that the equilibrium state for the Fokker-Planck equation on the torus satisfies a logarithmic Sobolev inequality, see for example \cite{G75}. Therefore we can substitute \ref{def:LS} into \eqref{eq:dissiq} to get
\begin{align}
\frac{\mathrm{d}}{\mathrm{d}t} \mathscr{F}(h) \leq & - \frac{b}{2}  \int_{\mathbb{T}^d \times \mathbb{R}^d} \frac{|\nabla_x h|^2 + |\nabla_v h|^2}{h} \mathrm{d}\mu - \frac{b}{2C_{LS}} \int_{\mathbb{T}^d \times \mathbb{R}^d} h \log(h) \mathrm{d}\mu\\
\leq & -C \mathscr{F}(h).
\end{align} We can then conclude by Gr\"{o}nwall's inequality that $\mathscr{F}$ is decreasing exponentially fast. This implies that both the relative entropy and relative Fisher information will decrease.

This proof relies on the fact that the dissipation of relative entropy for the kinetic Fokker-Planck is a Fisher Information type term. In the case of the linear relaxation Boltzmann equation this is no longer the case. It is still possible to generate 
\[ \int_{\mathbb{T}^d \times \mathbb{R}^d} \frac{|\nabla_v h|^2}{h} \mathrm{d}\mu, \] we get this from because for $h$ a solution to \eqref{eq:bigeqh} then
\begin{align*} \frac{\mathrm{d}}{\mathrm{d}t} \int_{\mathbb{T}^d \times \mathbb{R}^d} \frac{|\nabla_v h|^2}{h} \mathrm{d}\mu =& -\int_{\mathbb{T}^d \times \mathbb{R}^d} \frac{\nabla_xh \cdot \nabla_v h}{h} \mathrm{d}\mu \\
& - \int_{\mathbb{T}^d \times \mathbb{R}^d} \frac{|\nabla_v h|^2}{h} \left( 1+ \frac{\Pi h}{h} \right) \mathrm{d}\mu. \end{align*} However, because we also generate a term like
\[ \int_{\mathbb{T}^d \times \mathbb{R}^d} \frac{\nabla_x h \cdot \nabla_v h}{h} \mathrm{d}\mu.   \] This makes it impossible to close a Gr\"{o}nwall inequality purely on components of Fisher Information. We must expect this to be true since if we could close a Gr\"{o}nwall inequality on a functional without using a logarithmic Sobolev inequality or similar then exactly the same calculations would work for the equation where $x$ is in the whole space with no confining potential. In this situation we would not see exponential convergence of the relative entropy to zero. Therefore our strategy is to find another entropy type term whose dissipation is a Fisher Information type term. We are motivated by the fact that $\Pi h$ the macroscopic density will gain some regularity due to the averaging lemma, which says that the free transport will generate $H^{1/2}$ regularity for $\Pi h$, see for example \cite{GLP88}. Although we do not prove a regularizing result on $\Pi h$ and are results are not directly related we still find it useful to look at entropies involving $\Pi h$. In fact we will prove in section \ref{proofs} that for $h$ a solution to \eqref{eq:bigeqh}
\begin{equation}
\frac{\mathrm{d}}{\mathrm{d}t} \int_{\mathbb{T}^d \times \mathbb{R}^d} \Pi h \log \left( \Pi h \right) \mathrm{d}\mu = \int_{\mathbb{T}^d \times \mathbb{R}^d} \frac{ \Pi \left( \nabla_v h \right) \cdot \Pi \left( \nabla_x h \right)}{ \Pi h} \mathrm{d}\mu.
\end{equation} 
This is the key new idea used in our proof. The full strategy is explained in section \ref{proofs}.

 A similar problem occurs when working in $H^1$ norms. This situation is studied in \cite{NM06}. Here they work with $g=f/\sqrt{\mathcal{M}}$ instead of working with $h = f/\mathcal{M}$, this is only possible in the Hilbert space setting. In this situation we can write an equation on $g$

\begin{equation} \label{eq:bigeqg} \partial_t g + v \cdot \nabla_x g = \left( \int_{\mathbb{T}^d \times \mathbb{R}^d} g(t,x,u) \sqrt{\mathcal{M}(u)} \mathrm{d}u \sqrt{\mathcal{M}(v)}\right) - g(t,x,v) := \bar{\Pi} g - g. \end{equation} 

Working in terms of $g$ rather than $h$ means sacrificing simplicity in bounding the dissipation of $\| \nabla_v h \|^2_{L^2(\mu)}$ relative to $\|\nabla_v g\|^2_{L^2}$ for simplicity in controlling the mixed term. Precisely we have that for every $\delta >0$
\[ \frac{\mathrm{d}}{\mathrm{d}t} \| \nabla_v g\|^2_2 \leq -2 \langle \nabla_v g, \nabla_x g \rangle + \delta \| \nabla_v g \|^2_2 + C(\delta) \| g\|^2_2  \] and for every $\eta >0$
\[ \frac{\mathrm{d}}{\mathrm{d}t} \langle \nabla_v g, \nabla_x g \rangle \leq - \|\nabla_x g \|_2^2 + \eta \|\nabla_v g \|^2_2 + C(\eta) \| \nabla_x (\bar{\Pi} g - g) \|_2^2.  \] This means we can control mixed derivatives appearing in the dissipation of our functional up to producing a large amount of $\|g\|_2^2$ and this can be controlled by adding $\|g\|^2_2$ to the original functional. It is currently unclear whether it is possible to make a similar strategy work for relative entropy and Fisher information.



\subsection{Results}

 In fact we prove do not work only in relative entropy. We instead study general $\Phi$-entropies and $\Phi$-Fisher informations defined respectively by
\begin{align*}
H^\Phi =& \int_{\mathbb{R}^d \times \mathbb{T}^d} \Phi (h) \mathrm{d}\mu\\
I^\Phi =& \int_{\mathbb{R}^d \times \mathbb{T}^d} \Phi''(h)|\nabla h|^2 \mathrm{d}\mu.
\end{align*} We work with $\Phi$ a positive function such that $\Phi(1) = 0, \Phi''(t) > 0 \hspace{5pt} \forall t$, $1/\Phi''(t)$ a concave function and $\Phi(t)\Phi''(t) > 2 \Phi'(t)^2 \hspace{5pt} \forall t.$
\begin{defn}
We say that a measure $\mu$ satisfies a $\Phi$-logarithmic Sobolev inequality if there exists a constant $C>0$ such that for all $h$ with $\int f \mathrm{d}\mu = 1$ we have
\[ H^\Phi(h) \leq C I^\Phi(h). \]
\end{defn}

\begin{remark}
The conditions of $\Phi$ are satisfied when $\Phi$ is one of 
\[ \Phi_1(t) := t \log (t) -t +1 \] and
\[ \Phi_p (t) := \frac{1}{p-1}\big(t^p -1 - p(t-1)\big), \] where $p \in (1,2]$.
These quantities interpolate between the quadratic functional case $p=2$ which is the $L^2$ norm, and the Boltzmann entropy case, $p \sim 1$. They are used in \cite{ AMTU01, BG10} to study Fokker-Planck equations and convergence to equilibrium. Here we have inequalities due to Beckner in \cite{B89} which play the same role as the logarithmic Sobolev inequality does in showing hypocoercivity in Boltzmann entropy. They are of the form
\[ \int_{\mathbb{T}^d \times \mathbb{R}^d} \frac{h^p -h}{p(p-1)} \mathrm{d}\mu \leq C \int_{\mathbb{T}^d \times \mathbb{R}^d}  h^{p-2} |\nabla_{x,v} h|^2 \mathrm{d}\mu. \] These can be shown by interpolating between Poincar\'{e} and logarithmic Sobolev inequality \cite{ABD07}.
Beckner Inequalities are often stated in the form
\[ \int u^2 \mathrm{d}\mu - \left(\int u^p \mathrm{d}\mu \right)^{2/p} \leq (2-p) C \int |\nabla u|^2 \mathrm{d}\mu. \] It is straightforward to show that this equivalent to the form given above. (Write $q=2/p, h=u^p$ and assume by homogeneity that $\int h = 1$.)
\end{remark}

\begin{thm}\label{phitheorem} Let $\Phi$ satisfy the conditions in lemma \ref{phiconvex} and also let $\Phi$ be such that the uniform measure on the torus satisfies a $\Phi$-Sobolev inequality, $\Phi(t)\Phi''(t) > 2 \Phi'(t)^2 \hspace{5pt} \forall t$ and $1/\Phi''$ is a concave function.
If $f$ is a solution to (\ref{eq:bigeqh}) with initial data $h_0$ such that
\[ \int_{\mathbb{R}^d \times \mathbb{T}^d} \Phi''(h_0)|\nabla_{x,v} h_0|^2 \mathrm{d}\mu < \infty, \qquad  f_0 \in W^{1,1}(\mu), \] then there exist constants $\Lambda$ and $A$ depending on $\lambda$ and the constant in the $\Phi$-Sobolev inequality, such that
\[ I_\mu^\Phi(h_t) + H_\mu^\Phi(\Pi h_t) \leq A\exp \left( - \Lambda t\right) \Big(  I_\mu^\Phi(h_0) +  H_\mu^\Phi(\Pi h_0) \Big).\] This implies that if the equilibrium measure satisfies a $\Phi$-Sobolev inequality then for some $\gamma$,
\[ H(h_t) \leq \gamma \exp \left( - \Lambda t\right) I(h_0) .\]
We can take
\[ \Lambda = \min \left\{1, \frac{C}{4(1+\lambda)} \right\}\min\{2,\lambda/2\}  \] and
\[A = 4 \max \{2(1+1/\lambda)^2, (1+ \lambda)  \}.  \] Here $C$ is the constant in the $\Phi$-Sobolev inequality for the uniform measure on the torus.
\end{thm}

\subsection{Perspectives}
This work raises two natural questions. The first is whether a similar strategy can be used to show convergence to equilibrium for the linear relaxation Boltzmann equation when $x$ is in the whole space and the transport operator also involves a confining potential term. For the kinetic Fokker-Planck equation Villani shows convergence in $H^1$ and Boltzmann entropy in the first section of \cite{V09}. In \cite{M15} Monmarch\'{e} proves a general theorem which shows that hypocoercivity holds for the kinetic Fokker-Planck equation with confining potential in a class of $\Phi$ entropies which include the $p$-entropies. The situation is different for the linear relaxation Boltzmann equation. It is shown to be hypocoercive in $L^2$ in \cite{H07, DMS15}. In these works they use the inverse of an elliptic operator in order to create norms where you can compare the effect of the transport to the effect of the collisions. Emulating this strategy would be difficult in $\Phi$-entropies. To show hypocoercivity for the linear relaxation Boltzmann equation with a confining potential in $\Phi$-entropies would involve a very different strategy to our proofs in this equation. However, in the near to quadratic case it is possible to exploit additional cancellations happening in the operator to show convergence as is shown in \cite{M17} using calculations based on the original version of this paper.

The second natural question is whether this strategy could be extended to different collision operators which are also not regularising. For example the more complex scattering operators of the form
\[ \mathcal{C}(h)(x,v) = \int_{\mathbb{R}^d} \Big(k(u,v)h(u) - k(v,u)h(v)  \Big) \mathcal{M}(u) \mathrm{d}u, \] where
\[ \int_{\mathbb{R}^d} \Big( k(u,v)-k(v,u) \Big) \mathcal{M}(u) \mathrm{d}u = 0. \] Here our main goal would be the linear Boltzmann operator where
\[ \mathcal{C}(h)(x,v) = Q(h\mathcal{M}, \mathcal{M})\mathcal{M}(v)^{-1}. \] Where here $Q$ is the Boltzmann collision operator
\[ Q(f,g) = \int_{\mathbb{R}^d} \int_{\mathbb{S}^{d-1}} \Big(f(v')g(v_*') - f(v) g(v_*) \Big)\mathrm{d}\sigma\mathrm{d}v_*, \]
\[ v' = \frac{v+v_*}{2} +\frac{|v-v_*|}{2} \sigma, \quad v_*' = \frac{v+v_*}{2} - \frac{|v-v_*|}{2}\sigma. \] 

Exactly the same proof will work for a scattering operator which satisfies that for any positive definite constant matrix $S$ we have
\[\left( \frac{\mathrm{d}}{\mathrm{d}t} \right)_{\mathcal{C}} \int_{\mathbb{T}^d \times \mathbb{R}^d} \frac{\nabla_{x,v} h \cdot S \nabla_{x,v} h}{h} \mathrm{d}\mu \leq \int_{\mathbb{T}^d \times \mathbb{R}^d} \frac{\nabla_{x,v} \Pi h \cdot S \nabla_{x,v} \Pi h}{\Pi h} \mathrm{d}\mu  -\int_{\mathbb{T}^d \times \mathbb{R}^d} \frac{\nabla_{x,v} h \cdot S \nabla_{x,v} h}{h} \mathrm{d}\mu. \] Here $\Pi$ is still the projection onto the space of functions depending only on $x$. Unfortunately we are currently unable to do this for collision operators which are not straightforwardly comparable to the linear relaxation Boltzmann collision operator.

\subsubsection*{Acknowledgements}
I would like to thank Cl\'{e}ment Mouhot for pointing me towards this problem, suggesting I tried to emulate the techniques in \cite{NM06} and suggesting I look at other $\Phi$-entropies. I would also like to thank Pierre Monmarch\'{e} for many useful comments on the style, notation and references in an earlier draft of this paper. I also had several useful discussions with Tom Holding in the very early stages of this paper about possible forms for the derivatives of the $X$ part of the Fisher information.

\section{Proofs for General $\Phi$-entropy} \label{proofs}

Throughout the main parts of this chapter we work with an $h$ which is bounded above and below by constants and has bounded derivatives of all orders. In this set of possible $h$, all the integration by parts and differentiating through the integral are justified. In the appendix we show that these properties are propagated by the equation and that we can extend the result to a wider set using a density argument.

First we prove our entropies are well behaved. Lets define the functional
\[ J^\Phi_\mu(h) = \int_{\mathbb{R}^d \times \mathbb{T}^d} \Phi''(h)\Big(a |\nabla_x h|^2 + 2b \nabla_x h \cdot \nabla_v h + c |\nabla_v h|^2\Big) \mathrm{d}\mu. \]
\begin{lemma} \label{phiconvex}
Let $\Phi$ satisfy $\forall t > 0$:
\begin{itemize}
\item $\Phi(t) \geq 0$
\item $\Phi''(t) \geq 0$
\item $\Phi''(t) \Phi^{(4)}(t) > 2\Phi^{(3)}(t)^2 $
\end{itemize} Then if $b^2 \leq ac$ then $J$ is a convex functional.
\end{lemma}
\begin{proof}
Since $b^2 < ab$ we can write $J$ as the sum of functionals like
\[ \tilde{J}(h) = \int_{\mathbb{R}^d\times \mathbb{T}^d}\Phi''(h)|\alpha \nabla_x h + \beta \nabla_v h|^2\mathrm{d}\mu. \] Then the if the function
\[ \phi(\mathbf{x},y) = \Phi''(y) |\mathbf{x}|^2 \] the whole functional will be convex. This is because if $\phi$ is convex then
\begin{align*} \tilde{J}\Big(th+(1-t)g\Big) =& \int_{\mathbb{R}^d \times \mathbb{T}^d}\phi \Big(t(\alpha \nabla_x h + \beta \nabla_v h) + (1-t)(\alpha \nabla_x g + \beta\nabla_v g), th + (1-t) g\Big) \mathrm{d}\mu \\
\leq & \int_{\mathbb{R}^d \times \mathbb{T}^d}\Big( t\phi(\alpha\nabla_x h + \beta \nabla_v h, h) + (1-t)\phi(\alpha \nabla_x g + \beta \nabla_v g, g)\Big)\mathrm{d}\mu\\
=& t\tilde{J}(h) + (1-t)\tilde{J}(g).\end{align*} So we have reduced to showing that $\phi$ is convex. The function $\phi$ is the sum of functions $\tilde{\phi} = \Phi''(y)x^2$ where now $x$ is one dimensional. So we only need to show that these are convex. The Hessian of $\tilde{\phi}$ is
\[ \left( \begin{array}{c c} 2 \Phi''(y) & 2 x \Phi^{(3)}(y) \\ 2 x \Phi^{(3)}(y) & x^2 \Phi^{(4)}(y) \end{array}\right).\] This has positive trace as both diagonal terms are positive by our assumptions. It also has determinant $ 2x^2\Phi''(x) \Phi^{(4)}(x) - 4 x^2 \Phi^{(3)}(x)^2$ which is again positive due to he assumptions we made on $\Phi$ therefore the Hessian is positive definite so $\tilde{\phi}$ is convex. 
\end{proof}

We now outline our strategy for the proof. Our goal is to get constructive rates of convergence to equilibrium by closing a Gr\"{o}nwall estimate on a functional that we construct. This functional is composed from the components of Fisher information and an entropy term. In order to explain the strategy compactly we introduce the components of Fisher information.
\begin{align*}
I^X :=& I^X(h) =  \int_{\mathbb{R}^d \times \mathbb{T}^d} \Phi''(h)|\nabla_x h|^2 \mathrm{d} \mu ,\\
I^V :=& I^V(h) =  \int_{\mathbb{R}^d \times \mathbb{T}^d} \Phi''(h)|\nabla_v h|^2 \mathrm{d} \mu ,\\
I^M :=& I^M(h) =  \int_{\mathbb{R}^d \times \mathbb{T}^d} \Phi''(h)\nabla_x h \cdot \nabla_v h \mathrm{d} \mu. \\
\end{align*} We note here that $I^M$ does not have a sign. We also introduce a projected entropy which we use in our functional,
\[  H_{\Pi}(h) = \int_{\mathbb{T}^d} \Phi(\Pi h) \mathrm{d}x. \] We have another term which only appears in the intermediate steps of the proof,
\begin{align*}
I^{\Pi X} :=&  \int_{\mathbb{R}^d \times \mathbb{T}^d} \Phi''(\Pi h)|\nabla_x(\Pi h)|^2  \mathrm{d}\mu,
\end{align*} We prove later in this section that $I^X - I^{\Pi X} \geq 0$. 

By differentiating along the flow of the equation we show something close to the inequalities
\begin{align}
\frac{\mathrm{d}}{\mathrm{d}t} I^X \leq & - \lambda \left( I^X- I^{\Pi X} \right), \label{IX}\\
\frac{\mathrm{d}}{\mathrm{d}t} I^M \leq &-I^X - \lambda I^M , \label{IM} \\
\frac{\mathrm{d}}{\mathrm{d}t} I^V \leq  &-2I^M - \lambda I^V. \label{IV}
\end{align} Actually there are extra elements appearing which would cancel out when these terms are combined into the type of functional we look at so these inequalities are not quite true. In fact we prove a global inequality on a functional like $J$ defined bellow but it is clearer to separate the elements here.
We begin by constructing a functional of the form 
\[ J = a I^X + 2b I^M + c I^V, \] with $a c - b^2 > 0.$ This inequality means that $J$ is equivalent to the Fisher information $I$.

 We now give a strategy for choosing $a,b,c$. We need that $b$ is non-zero since inequality \ref{IM} provides the negative $I^X$ which we want in the derivative. The most natural next step would be to use the Cauchy-Schwartz inequality to control $I^M$ by $I^X$, and $I^V$. However, we can check that the quantity of $I^M$ is too large for this to be possible. We need to utilise inequality \eqref{IX}. We do this by showing that
\begin{align} - I^M \leq \frac{\eta}{2}I^V + \frac{1}{2 \eta} \left( I^X - I^{\Pi X} \right) - \frac{\mathrm{d}}{\mathrm{d}t} H_\Pi. \label{H}\end{align} This is the key new element in our proof.

By adding a quantity of $H_\Pi$ to the functional and using inequality \eqref{H}, we can now control $I^M$ by $I^V$ and $I^X - I^{\Pi X}$. Since the inequality \eqref{IX} doesn't produce bad terms we are free to add as much $I^X$ to the functional as we need. Therefore, by adding a large amount of $H_\Pi$ and $I^X$ to our functional we can cancel out the positive $I^X - I^{\Pi X}$. Therefore we can make $\eta$ small. This means the sum of the positive $I^V$ from controlling $I^M$ and the negative $I^V$ from inequality \eqref{IV} will sum to a negative amount of $I^V$. We recall that we also have some negative $I^X$ for inequality \eqref{IM}. So we have,
\[ \frac{\mathrm{d}}{\mathrm{d}t} (J + A_4 H_\Pi) \leq - C (I^X + I^V). \] We then use the equivalence between $J$ and $I$ and the logarithmic Sobolev inequality to get
\[ \frac{\mathrm{d}}{\mathrm{d}t}( J+ A_4 H_\Pi) \leq - C(J+ A_4 H_\Pi). \] So we can close a Gronwall estimate and then use the equivalence between $J$ and $I$ again to translate this to an inequality on $I$.

In order to prove our theorem we would like to study how a functional like $J$ behaves under the action of the collision part of the operator. We write $L= \lambda (\Pi - I)$ and $T= -v\cdot \nabla_x$ and write $(\mathrm{d}/\mathrm{d}t)_O$ to write the derivative along the flow of the operator $L$. We have that
\begin{lemma} \label{phiflowL}
We can differentiate $J$ along the flow of $L$ to get that
\begin{align*}\left( \frac{\mathrm{d}}{\mathrm{d}t} \right)_L J_\mu^\Phi(h) \leq & a \left(\int_{\mathbb{R}^d\times \mathbb{T}^d} \Phi''(\Pi h) |\nabla_x \Pi h|^2 \mathrm{d}\mu- \int_{\mathbb{R}^d \times \mathbb{T}^d} \Phi'(h) |\nabla_x h|^2 \mathrm{d}\mu \right)\\
&- 2b \int_{\mathbb{R}^d \times \mathbb{T}^d} \Phi''(h) \nabla_x h \cdot \nabla_v h \mathrm{d}\mu - c \int_{\mathbb{R}^d \times \mathbb{T}^d}\Phi''(h)|\nabla_v h|^2 \mathrm{d}\mu.  \end{align*}
\end{lemma}
\begin{proof}
As $J_\mu^\Phi$ is convex we can see by Taylor expanding that
\begin{align*}
J_\mu^\Phi(e^{Ls}h(t)) = & J_\mu^\Phi\left(h(t)+ \lambda s (\Pi - I) h(t) + o(s)\right)\\
\leq & (1-\lambda s)J_\mu^\Phi\left(h(t)+o(s) \right) + \lambda s J_\mu^\Phi(\Pi h(t)).
\end{align*} Now we calculate that
\begin{align*}
J_\mu^\Phi(\Pi h) =& \int_{\mathbb{R}^d \times \mathbb{T}^d}\Phi''(\Pi h)\Big(a|\nabla_x \Pi h|^2 + 2b \nabla_x \Pi h \cdot \nabla_v \Pi h + c |\nabla_v \Pi h|^2 \Big)\mathrm{d}\mu\\
= & a\int_{\mathbb{R}^d \times \mathbb{T}^d} \Phi''(\Pi h)|\nabla_x \Pi h|^2 \mathrm{d}\mu.
\end{align*}
This means that 
\begin{align*}
J_\mu^\Phi(e^{sL}h(t)) - J_\mu^\Phi(h(t)) \leq & \lambda s a\left(\int_{\mathbb{R}^d \times \mathbb{T}^d}\Phi''(\Pi h)|\nabla_x \Pi h|^2 \mathrm{d} \mu - \int_{\mathbb{R}^d \times \mathbb{T}^d} \Phi''(h) |\nabla_x h|^2 \mathrm{d}\mu \right)\\
&- \lambda s b \int_{\mathbb{R}^d \times \mathbb{T}^d} \Phi''(h) \nabla_x  h \cdot \nabla_v h \mathrm{d} \mu \\ &- \lambda s c \int_{\mathbb{R}^d \times \mathbb{T}^d} \Phi''(h) |\nabla_v h|^2 \mathrm{d}\mu \\
& + J_\mu^\Phi(h(t) + o(s)) - J_\mu^\Phi (h(t)).
\end{align*} Dividing by $s$ and taking the limit as $s \rightarrow 0$ gives the result.
\end{proof}
We now need to look at how $J$ behaves under the flow of $T$.
\begin{lemma}\label{phiflowT}
We can differentiate $J$ along the flow of $T$ to get that
\begin{align*}
\left(\frac{\mathrm{d}}{\mathrm{d}t}  \right)_T J_\mu^\Phi(h) = & -2b \int_{\mathbb{R}^d \times \mathbb{T}^d}\Phi''(h) |\nabla_x h|^2\mathrm{d}\mu -2c\int_{\mathbb{R}^d \times \mathbb{T}^d}\Phi''(h) \nabla_x h \cdot \nabla_v h \mathrm{d}\mu.
\end{align*}
\end{lemma}
\begin{proof}
This is just a simple application of the chain rule. We have
\begin{align*}
\left(\frac{\mathrm{d}}{\mathrm{d}t}  \right)_TJ_\mu^\Phi(h) = & -  a \int_{\mathbb{R}^d \times \mathbb{T}^d}\Phi'''(h)(v \cdot \nabla_x h)|\nabla_x h|^2 \mathrm{d}\mu -2 a \int_{\mathbb{R}^d \times \mathbb{T}^d}\Phi''(h) \nabla_x(v \cdot \nabla_x h)\cdot \nabla_x h \mathrm{d}\mu \\
& -2b \int_{\mathbb{R}^d \times \mathbb{T}^d} \Phi'''(h)(v \cdot \nabla_x h)\nabla_x h\cdot \nabla_v h \mathrm{d}\mu - 2b \int_{\mathbb{R}^d \times \mathbb{T}^d} \Phi''(h) \nabla_x(v \cdot \nabla_x h) \cdot \nabla_v h \mathrm{d}\mu\\
& - 2b \int_{\mathbb{R}^d \times \mathbb{T}^d} \Phi''(h) \nabla_x h \cdot \nabla_v(v \cdot \nabla_x h) \mathrm{d}\mu - c \int_{\mathbb{R}^d \times \mathbb{T}^d} \Phi'''(h)(v \cdot \nabla_x h) |\nabla_v h|^2 \mathrm{d}\mu \\
& - 2c \int_{\mathbb{R}^d \times \mathbb{T}^d} \Phi''(h) \nabla_v (v \cdot \nabla_x h) \cdot \nabla_v h \mathrm{d}\mu\\
= & \int_{\mathbb{R}^d \times \mathbb{T}^d} v \cdot \nabla_x \left(\Phi''(h)(a |\nabla_x h|^2 + 2b \nabla_x h \cdot \nabla_v h + c |\nabla_v h|^2  \right) \mathrm{d}\mu \\
& - 2b \int_{\mathbb{R}^d \times \mathbb{T}^d} \Phi''(h) |\nabla_x h|^2 \mathrm{d}\mu - 2c \int_{\mathbb{R}^d \times \mathbb{T}^d} \Phi''(h) \nabla_x h \cdot \nabla_v h \mathrm{d}\mu\\
=& -2b\int_{\mathbb{R}^d \times \mathbb{T}^d}\Phi''(h) |\nabla_x h|^2 \mathrm{d}\mu - 2c \int_{\mathbb{R}^d \times \mathbb{T}^d} \Phi''(h) \nabla_x h \cdot \nabla_v h \mathrm{d}\mu.
\end{align*}
\end{proof}
Now we need to show our helpful lemma relating projected entropy to the mixed term. This result relates the quantities involving only $\Pi h$ to quantities coming from the full Fisher information. For this we define the local average speed $U(x)$, of a solution to (\ref{eq:bigeqf}) by
\[ U(x) := \int_{\mathbb{R}^d} v h(v,x) \mathcal{M}(v) \mathrm{d}v = \int_{\mathbb{R}^d} v f(v,x) \mathrm{d}v.  \]
\begin{lemma}\label{Jensen}
Suppose that the uniform measure on the torus satisfies a $\Phi$-Sobolov inequality. Then for any $h$ we have that
\[ I^{\Pi X}(h) = \int_{\mathbb{R}^d \times \mathbb{T}^d} \Phi''(\Pi h)|\nabla_x \Pi h |^2 \mathrm{d} \mu \leq \int_{\mathbb{R}^d \times \mathbb{T}^d} \Phi''(h)|\nabla_xh|^2 \mathrm{d} \mu. \] This implies that for all $h$ there exists a constant $C$ such that
\[ H_{\Pi}(h) = \int_{\mathbb{T}^d} \Phi(\Pi h) \mathrm{d}x \leq C \int_{\mathbb{T}^d\times \mathbb{R}^d} \Phi''( h)|\nabla_x h|^2 \mathrm{d} \mu. \]
Finally, if $h$ is a solution to (\ref{eq:bigeqh}) then 
\[ \frac{\mathrm{d}}{\mathrm{d}t} H_{\Pi}(h(t)) = - \int_{\mathbb{T}^d} \Phi'(\Pi h) \nabla_x \cdot U(x) \mathrm{d}x.  \]
\end{lemma}
\begin{proof}
We can see that the first inequality will follow from
\[ \Phi''(\Pi h) |\nabla_x \Pi h|^2 \leq \Pi \left(\Phi''(h) |\nabla_x h|^2 \right).  \] Since $\Pi$ is integrating against a probability measure we would like to use Jensen's inequality. Instead of looking at $h$ we consider $\mathbf{h} = (\nabla_x h, h)$ we have already shown the function $\phi(\mathbf{x},y) = \Phi''(y)|\mathbf{x}|^2$ is convex so from Jensen's inequality we have
\[ \phi(\Pi \mathbf{h}) \leq \Pi (\phi (\mathbf{h})). \] This implies our desired result since $\Pi$ commutes with $\nabla_x$. (Here $\Pi$ acts component wise on vectors).

Now since we have a $\Phi$-Sobolev inequality for the uniform measure on the torus we have
\[\int_{\mathbb{T}^d} \Phi(\Pi h) \mathrm{d}x \leq C \int_{\mathbb{T}^d} \Phi''(\Pi h) |\nabla_x \Pi h|^2 \mathrm{d}x.  \] We can then conclude this part by the first inequality. 

For the last part,
\begin{align*}
 \partial_t \Pi h = &- \int_{\mathbb{R}^d} v \nabla_x h \mathcal{M}(v) \mathrm{d}v  + \lambda \Pi(\Pi h) - \lambda \Pi h \\
 =& - \nabla_x \cdot U(x).
\end{align*}
 This implies that
\[ \partial_t H_\Pi = \int_{\mathbb{T}^d} \Phi'(\Pi h) \partial_t \Pi h \mathrm{d}x  = - \int_{\mathbb{T}^d} \Phi'(\Pi h) \nabla_x \cdot U(x) \mathrm{d}x.  \]
\end{proof}
We now need a lemma which will help us control the mixed derivative. 
\begin{lemma} \label{phimixedterm}
If $1/\Phi''(t)$ is a concave function then for any positive $\eta$ we have
\begin{align*}
- \int_{\mathbb{R}^d \times \mathbb{T}^d}\Phi''(h) \nabla_x h \cdot \nabla_v h \mathrm{d}\mu \leq & \frac{\eta}{2} \int_{\mathbb{R}^d \times \mathbb{T}^d} \Phi''(h) |\nabla_v h|^2 \mathrm{d} \mu \\& + \frac{1}{2\eta} \left( \int_{\mathbb{R}^d \times \mathbb{T}^d} \Phi''(\Pi h) |\nabla_x \Pi h|^2 \mathrm{d}\mu -  \int_{\mathbb{R}^d \times \mathbb{T}^d} \Phi''(h) |\nabla_x h|^2 \mathrm{d}\mu \right) \\ &- \frac{\mathrm{d}}{\mathrm{d}t}  \int_{\mathbb{R}^d \times \mathbb{T}^d} \Phi(\Pi h) \mathrm{d}\mu.
\end{align*}
\begin{proof}
We need to rewrite the mixed term
\begin{align*}
- \int_{\mathbb{R}^d \times \mathbb{T}^d} \Phi''(h) \nabla_x h \cdot \nabla_v h \mathrm{d}\mu = & - \int_{\mathbb{R}^d \times \mathbb{T}^d} \nabla_x \Phi'(h) \cdot \nabla_v h \mathrm{d}\mu \\
=&  -\int_{\mathbb{R}^d \times \mathbb{T}^d} \left( \nabla_x \Phi'(h) - \nabla_x \Phi'(\Pi h) \right)\cdot \nabla_v h\mathrm{d}\mu \\
& -  \int_{\mathbb{R}^d \times \mathbb{T}^d} \nabla_x \Phi'(\Pi h) \cdot \nabla_v h \mathrm{d}\mu \\
 \leq & \frac{\eta}{2}  \int_{\mathbb{R}^d \times \mathbb{T}^d} \Phi''(h) |\nabla_v h|^2 \mathrm{d}\mu \\
 & + \frac{1}{2\eta}  \int_{\mathbb{R}^d \times \mathbb{T}^d} \frac{|\nabla_x \Phi'(h) - \nabla_x \Phi'(\Pi h)|^2}{\Phi''(h)}\mathrm{d}\mu \\
 & -  \int_{\mathbb{R}^d \times \mathbb{T}^d} \Phi'(\Pi h) \nabla_x \cdot U(x) \mathrm{d}\mu.
\end{align*}
We get the equality for the last term since
\[ -\int \nabla_v h \mathcal{M}(v) \mathrm{d}v = - \int v h \mathcal{M}(v) \mathrm{d}v  = U(x). \] Then we can use the last part of lemma \ref{Jensen}. Now we observe that
\begin{align*}
 \int_{\mathbb{R}^d \times \mathbb{T}^d} \frac{|\nabla_x \Phi'(h) - \nabla_x \Phi'(\Pi h)|^2}{\Phi''(h)}\mathrm{d}\mu = &  \int_{\mathbb{R}^d \times \mathbb{T}^d} \Phi''(h) |\nabla_x h|^2 \mathrm{d}\mu \\
 & - 2  \int_{\mathbb{R}^d \times \mathbb{T}^d} \Phi''(\Pi h)\nabla_x h \cdot \nabla_x \Pi h \mathrm{d}\mu \\
 & +  \int_{\mathbb{R}^d \times \mathbb{T}^d} \frac{\Phi''(\Pi h)^2}{\Phi''(h)}|\nabla_x \Pi h|^2 \mathrm{d}\mu.
\end{align*} Now we see in the second term the only part which depends on $v$ is the $\nabla_x h$ so we can replace it by $\nabla_x \Pi h$. The last term is positive and the only term which depends on $v$ is $1/\Phi''(h)$ since we have that $1/\Phi''(h)$ is a concave function we have
\[ \Pi \left(\frac{1}{\Phi''(h)} \right) \leq \frac{1}{\Phi''(\Pi h)}.   \] Therefore we have that
\begin{align*}
\int_{\mathbb{R}^d \times \mathbb{T}^d} \frac{|\nabla_x \Phi'(h) - \nabla_x \Phi'(\Pi h)|^2}{\Phi''(h)}\mathrm{d}\mu \leq \int_{\mathbb{R}^d \times \mathbb{T}^d} \Phi''(h) |\nabla_x h|^2 \mathrm{d}\mu - \int_{\mathbb{R}^d \times \mathbb{T}^d} \Phi''(\Pi h) |\nabla_x \Pi h|^2 \mathrm{d}\mu.
\end{align*} This completes the proof of our lemma.
\end{proof}
\end{lemma}
Now we can prove the main theorem
\begin{proof}[Proof of Theorem \ref{phitheorem}] 
Using lemmas \ref{phiflowL} and \ref{phiflowT} we get that
\begin{align*}
\frac{\mathrm{d}}{\mathrm{d}t}J_\mu^\Phi(h) \leq & -2b \int_{\mathbb{R}^d \times \mathbb{T}^d} \Phi''(h) |\nabla_x h|^2 \mathrm{d}\mu - c \lambda \int_{\mathbb{R}^d \times \mathbb{T}^d} \Phi''(h) |\nabla_v h|^2 \mathrm{d}\mu \\
& - 2(b \lambda + c)\int_{\mathbb{R}^d \times \mathbb{T}^d} \Phi''(h) \nabla_x h \cdot \nabla_v h \mathrm{d}\mu \\
& + a\lambda \left(\int_{\mathbb{R}^d \times \mathbb{T}^d} \Phi''(\Pi h) |\nabla_x \Pi h|^2\mathrm{d}\mu - \int_{\mathbb{R}^d \times \mathbb{T}^d} \Phi''(h) |\nabla_x h|^2 \mathrm{d}\mu \right).
\end{align*} We now use lemma \ref{phimixedterm} to bound the mixed term.
\begin{align*}
\frac{\mathrm{d}}{\mathrm{d}t}J_\mu^\Phi(h) \leq & - 2b \int_{\mathbb{R}^d \times \mathbb{T}^d} \Phi''(h) |\nabla_x h|^2 \mathrm{d}\mu \\
& - \left(c \lambda - \eta (b\lambda+c) \right) \int_{\mathbb{R}^d \times \mathbb{T}^d} \Phi''(h) |\nabla_v h|^2 \mathrm{d}\mu \\
& - \left(a\lambda - \frac{1}{\eta}(b\lambda+ c) \right)\left(\int_{\mathbb{R}^d \times \mathbb{T}^d} \Phi''(\Pi h) |\nabla_x \Pi h|^2\mathrm{d}\mu - \int_{\mathbb{R}^d \times \mathbb{T}^d} \Phi''(h) |\nabla_x h|^2 \mathrm{d}\mu \right) \\
& - 2(\lambda b+c)\frac{\mathrm{d}}{\mathrm{d}t} \int_{\mathbb{R}^d \times \mathbb{T}^d} \Phi (\Pi h) \mathrm{d} \mu.
\end{align*} Now lets choose $a =2(1+1/\lambda)^2, b=1, c=1, \eta = \lambda/(2(\lambda+1))$. This gives
\begin{align*}
\frac{\mathrm{d}}{\mathrm{d}t}J_\mu^\Phi(h) \leq & - 2 \int_{\mathbb{R}^d \times \mathbb{T}^d} \Phi''(h) |\nabla_x h|^2\mathrm{d}\mu - \frac{\lambda}{2} \int_{\mathbb{R}^d \times \mathbb{T}^d} \Phi''(h) |\nabla_v h|^2 \mathrm{d}\mu \\
& - 2(\lambda+1)\frac{\mathrm{d}}{\mathrm{d}t} \int_{\mathbb{R}^d \times \mathbb{T}^d} \Phi(\Pi h) \mathrm{d}\mu.
\end{align*}
This gives that
\begin{align*}
\frac{\mathrm{d}}{\mathrm{d}t} (J_\mu^\Phi(h) + 2(\lambda+1)H_\mu^\Phi( \Pi h) ) \leq &- \frac{1}{2}\min \left\{ 2, \lambda/2 \right\} \int_{\mathbb{R}^d \times \mathbb{T}^d} \Phi''(h) |\nabla h|^2 \mathrm{d}\mu\\ & - \frac{1}{2}C\min \left\{2, \lambda/2 \right\} \int_{\mathbb{R}^d \times \mathbb{T}^d} \Phi(\Pi h)\mathrm{d}\mu.
\end{align*}
Since we have that $2(1+1/\lambda)^2X^2 + XV +V^2 \geq (X^2+V^2)/2$ this means
\begin{align*} \frac{\mathrm{d}}{\mathrm{d}t} (J_\mu^\Phi(h) + 2(1+\lambda) H_\mu^\Phi( \Pi h) ) \leq &-\min \{2, \lambda/2 \} \left(J_\mu^\Phi(h) - \frac{1}{4(\lambda +1)}C2(\lambda+1)H_\mu^\Phi(\Pi h)\right)  \\
\leq & - \min \left\{1, \frac{C}{4(\lambda +1)}  \right\}\min \{2, \lambda/2\}\left(J_\mu(h) + 4 H_\mu^\Phi(\Pi h) \right).\end{align*} Therefore if we set
\[ \Lambda = \min \left\{1, \frac{C}{4(\lambda+1)}  \right\}\min \{2, \lambda/2 \}, \] we have that
\[ J_\mu^\Phi(h(t)) + 2(1+ \lambda) H_\mu^\Phi(\Pi h(t)) \leq e^{-\Lambda t} \left( J_\mu^\Phi(h(0)) + 2(1+ \lambda) H_\mu^\Phi(\Pi h(0))\right). \]
Now we use that for all $h$,
\[\frac{1}{2}I_\mu^\Phi(h) \leq J_\mu^\Phi(h) \leq 4(1+1/\lambda)^2 I_\mu^\Phi(h).  \] This means that
\begin{align*}
I_\mu^\Phi(h(t)) + H_\mu^\Phi(\Pi h(t)) \leq & 2 \left( J_\mu^\Phi(h(t)) + 2(1+ \lambda) H_\mu^\Phi(\Pi h(t))\right) \\
\leq & 2 e^{-\Lambda t} \left( J_\mu^\Phi(h(0)) + 2(1+ \lambda) H_\mu^\Phi(\Pi h(0))\right) \\
\leq & 2e^{-\Lambda t} \left(4(1+1/\lambda)^2 I_\mu^\Phi(h(0)) + 2(1+ \lambda) H_\mu^\Phi (\Pi h(0)) \right)\\
\leq & 4 \max \left\{2(1+1/\lambda)^2, (1+ \lambda) \right\}e^{-\Lambda t} \left(I_\mu^\Phi (h(0)) + H_\mu^\Phi (\Pi h(0)) \right).
\end{align*}
\end{proof}

\appendix
\section{}

We show for $h$, being bounded above and bellow and having bounded derivatives of all orders is propagated by the equation (this is similar to what is shown in the appendix of \cite{CCG03}). In this set we can do all the calculations given in the main part of the paper. We then show for $h \in W^{1,1}(\mu)$ with finite Fisher information then we can make a density argument to show that the result still holds in this case.

\begin{lemma}
The equation preserves bounded derivatives of all orders.
\end{lemma}
\begin{proof}

We rewrite the equation for $h$ in a mild formulation as follows
\[ e^{\lambda t} h(t,x,v) = h(0, x-vt, v) + \lambda \int_0^t e^{\lambda s} \int h(s, x-v(t-s), u) \mathcal{M}(u) \mathrm{d}u \mathrm{d}s. \]
This leads to the following inequality
\[ e^{\lambda t} \| D^{\alpha}_x h(t) \|_{\infty} \leq \| D^{\alpha}_x h(0) \| + \lambda \int_0^t e^{\lambda s} \| D^{\alpha}_x h(s) \|_{\infty} \mathrm{d}s. \]
 Therefore by Gronwall's inequality  we have that
\[ \| D^{\alpha}_x h(t) \|_{\infty} \leq \|D_x^{\alpha} h(0)\|_{\infty}. \] We also from this mild formulation that any mixed derivative can be written in terms of $x$ derivative and derivatives of the initial data. Therefore, the derivatives will remain in $L^\infty$ for all time.
\end{proof}

\begin{lemma}
The equation preserves positivity and constants are a steady state of the equation therefore being bounded above and below is preserved.
\end{lemma}
\begin{proof}
We can show that
\[ \partial_t \left( e^{\lambda t}h(t, x+vt,v) \right) = \int \lambda e^{\lambda t} h(t, x+vt, u) \mathcal{M}(u) \mathrm{d}u. \] Therefore if $e^{\lambda t}h(t,x+vt,v)$ is positive for all $x$ and $v$ then so is its derivative. Therefore it will remain positive for all time.

 It is easy to check that constants are a steady state so if $h(0)-c$ is positive then since positivity is preserved so is $h(t)-c$ and similarly if $C-h(0)$ is positive then so is $C-h(t)$.
\end{proof}


\begin{lemma} Suppose that we have $h(0)$ is in $W^{1,1}(\mu)$ with bounded Fisher information, and also suppose we have a sequence $h_n(0)$ which is bounded above and below, has bounded derivatives up to second order and converges to $h(0)$ in $L^1(\mu)$ with 
\[ H^{\Phi}(h_n(t)) \leq Ae^{-\Lambda t}I^{\Phi}(h_n(0)), \] for every $n$ then we have
\[ H(h(t)) \leq Ae^{-\Lambda t} I(h(0)). \]
\end{lemma}

\begin{proof}
Convergence in $L^1$ implies that $h_n$ tends to $h$ a.e. along a subsequence. Also, suppose that $h_1$ and $h_2$ are two solutions to the equation then
\[ \sup_{s \leq t}\|h_1(s) -h_2(s)\|_{L^1(\mu)} \leq e^{-\lambda t}\|h_1(0)-h_2(0)\|_{L^1(\mu)} + \sup_{s \leq t}\|h_1(s) - h_2(s)\|_{L^1(\mu)}(1-e^{-\lambda t}).  \] Therefore,
\[\sup_{s \leq t} \|h_1(s) - h_2(s) \|_{L^1(\mu)} \leq \| h_1(0) - h_2(0)\|_{L^1(\mu)},\] hence $h_n(t)$ tends to $h(t)$ in $L^1$ therefore $h_n(t)$ also converges to $h(t)$ almost everywhere along a subsequence.

Then since $\Phi(h_n(t)) \geq 0$ by Fatou's lemma we have
\[ \int\Phi(h(t))\mathrm{d}\mu \leq \liminf_n \int\Phi(h_n(t)) \mathrm{d}\mu. \]
Therefore, if we have $h$ a solution to the equation with initial data $h(0)$ as defined above we have that
\[ H^\Phi (h(t)) \leq \liminf_n Ae^{-\Lambda t} I^\Phi (h_n(0)). \] So to prove our theorem holds in this larger set it remains to show that we can find a sequence $h_n(0)$ converging to $h(0)$ in $L^1(\mu)$ where for every $n$ $h_n(0)$ is positive, integrates to 1 against $\mu$, is bounded bellow and has derivatives bounded of all orders which also satisfies
\[ \liminf_n I^\Phi (h_n(0)) \leq I^\Phi (h(0)). \] 

To do this we make a very standard molifier argument. Let $\chi$ be a smooth function on $\mathbb{R}_+$ with $\chi(x) =1$ for $x < 1$ and $\chi(x) =0$ for $x > 2$ and $\Phi''(\chi(x))|\chi'(x)|^2$ integrable. Then define $\chi_R(x,v) = \chi(\|v\|/R).$ Also let $\phi$ be a molifier integrating to one and compactly supported in $B(0,1)$ then set $\phi_\epsilon (x,v) = \epsilon^{-2d} \phi((x,v)/\epsilon).$ Take some $h$ in $W^{1,1}(\mu)$ with finite $\Phi$-Fisher information. Let $h_R = h \chi_R$, then set $h_{\epsilon, R} = \phi_\epsilon \star h_R$ and then $h_{\eta, \epsilon, R} = (h_{\epsilon, R} + \eta)/(\|h_{\epsilon, R}\|_1 + \eta)$. So $h_{\eta, \epsilon, R}$ is bounded below and has derivatives bounded of all orders and fairly clearly converges to $h$ in $L^1(\mu)$.

So first we try and get rid of $\eta$ since $\nabla h_{\eta, \epsilon, R} = \nabla h_{\epsilon,R}/(\|h_{\epsilon, R}\|_1 + \eta)$ we get that
\[ \Phi''(h_{\eta, \epsilon, R}) |\nabla h_{\eta, \epsilon, R}|^2 \hspace{5pt}\mbox{increases to}\hspace{5pt} \Phi''(h_{\epsilon, R})|\nabla h_{\epsilon, R} |^2. \] Therefore, by monotone convergence,
\[ \lim_{\eta \rightarrow 0} I^\Phi(h_{\eta, \epsilon, R}) = I^\Phi(h_{\epsilon, R}). \] 

Now we work on $\epsilon$, we have that $\nabla h_{\epsilon, R} = \phi_{\epsilon} \star \nabla h_R$. We can now make a similar argument based on Jensen's inequality and the fact that $\Phi''(y)|x|^2$ is convex to get that
\[ \Phi''(h_{\epsilon, R})|\nabla h_{\epsilon, R}|^2 \leq \phi_{\epsilon} \star \left( \Phi''(h_R)|\nabla h_R|^2 \right). \] Since, the mollification of and $L^1$ function converges in $L^1$ to that function we get that
\[ \lim_{\epsilon \rightarrow 0} I^\Phi(h_{\epsilon, R}) \leq I^\Phi(h_R). \]

Now we work on $R$, we note that 
\[ \Phi''(h_R)|\nabla h_R|^2 = \Phi''(h\chi_R)\left( \chi_R^2|\nabla h|^2 + \chi_R h \nabla h \cdot \nabla \chi_R +h^2 |\nabla \chi_R|^2\right) \] Since, $h, \nabla h, \Phi''(h)|\nabla h|^2$ are all in $L^1(\mu)$ we can see that
\[ \lim_{R \rightarrow \infty} I^\Phi(h_R) = I^\Phi(h). \]
\end{proof}
\bibliographystyle{abbrv}
\bibliography{kacbibliography.tex}{}
\end{document}